\documentclass[12pt]{amsart}

\usepackage{bbm}
\usepackage{mathrsfs}
\usepackage{amsfonts}
\usepackage{amsmath}
\usepackage{amssymb}

\usepackage{graphicx}

\usepackage{latexsym}
\usepackage{cases}
\usepackage{amscd}
\usepackage{amsfonts}

\usepackage[all]{xy}
\usepackage{xy}

\usepackage{amsfonts,amsmath, amssymb}

\newcommand{\triplet}{\mathcal{W}(p)}

\newtheorem{theorem}{Theorem}[section]

\newtheorem{lemma}[theorem]{Lemma}
\newtheorem{proposition}[theorem]{Proposition}
\newtheorem{remark}[theorem]{Remark}
\newtheorem{definition}[theorem]{Definition}

\newtheorem{convention}[theorem]{Convention}
\begin{document}

\title{Fusion rules of Virasoro Vertex Operator Algebras}

\author {Xianzu Lin }

\date{ }
\maketitle
   {\small \it College of Mathematics and Computer Science, Fujian Normal University, }\\
    \   {\small \it Fuzhou, {\rm 350108}, China;}\\
      \              {\small \it Email: linxianzu@126.com}
%\vskip 0.1cm \arraycolsep1.5pt
%\newtheorem{Lemma}{Lemma}[section]
%\newtheorem{Theorem}{Theorem}[section]
%\newtheorem{Definition}{Definition}[section]
%\newtheorem{Proposition}{Proposition}[section]
%\newtheorem{Remark}{Remark}[section]
%\newtheorem{Corollary}{Corollary}[section]
\begin{abstract}

In this paper we prove the fusion rules of Virasoro vertex operator algebras $L(c_{1,q},0)$, for $q\geq1$.
Roughly speaking, we consider $L(c_{1,q},0)$ as the limit of $L(c_{n,nq-1},0)$, for $n\rightarrow\infty$,
and the fusion rules of $L(c_{1,q},0)$ follow as the limits of the fusion rules of $L(c_{n,nq-1},0)$.

\


\

2000 MR Subject Classification: 17B69
\end{abstract}

%%%%%%%%%%%%%%%%%%%%%%%%%%%%%%%%%%%%%%%%%%%%%%%%%%%%%%%%%%%%%%%%%%%%%%%%%%%%%%%%%%%%%%%%%%%%%%%%%%
\section{Introduction}
Among classical representation theory (of compact groups or semi-simple Lie algebras),
the most important problems are,

\begin{enumerate}
\item The classification problem: describe all the irreducible representations.
\item The Clebsch-Gordon problem: given irreducible representations $V$ and $W$,
describe the decomposition, with multiplicities, of the representation $V\otimes W$.
\end{enumerate}
For the representation theory of vertex operator algebras, the most important problem is also the classification of all the irreducible representations.
The difference is that for two irreducible modules $U$ and $V$ over a vertex operator algebra $A$, we can not define the tensor module of $U$ and $V$.
Nevertheless, we still have the analogue of the Clebsch-Gordon problem via the notion of intertwining operators.
In particular, for three irreducible modules $U$, $V$ and $W$ over a vertex operator algebra $A$,
we can define the fusion rule $\mathcal {N}^{W}_{U,V}$, the analogue of the Clebsch-Gordon coefficient.
As in the classical representation theory, the second most important problem in representation theory of vertex operator algebras is to
determine the fusion rules $\mathcal {N}^{W}_{U,V}$.

The Virasoro vertex operator algebras constitute one of the most important classes of vertex operator algebras.
In \cite{w} it was proved that the vertex operator algebras $L(c_{p,q},0)$ are rational,
where $(p,q)=1$, $p,q>1$ and $c_{p,q}=13-6(\frac{q}{p}+\frac{p}{q})$.
 Furthermore, the fusion rules of $L(c_{p,q},0)$ were proved in \cite{w} using the Frenkel-Zhu's formula (cf.\cite{fz}).
 In the case of $L(c_{1,q},0)$, we cannot prove the fusion rules as in \cite{w},
 for Frenkel-Zhu's formula cannot be applied to $L(c_{1,q},0)$ which is non-rational.
  The fusion rules of $L(c_{1,1},0)$ were first proved in \cite{mi} and further extended in \cite{dj1}.

In this paper we prove the fusion rules of $L(c_{1,q},0)$ for all $q\geq1$.
Our method is totally different from those of \cite{dj1,mi}.
Roughly speaking, we consider $L(c_{1,q},0)$ as the limit of $L(c_{n,nq-1},0)$, for $n\rightarrow\infty$,
and the fusion rules of $L(c_{1,q},0)$ follow as the limits of the fusion rules of $L(c_{n,nq-1},0)$. Formally, the fusion ring of $L(c_{1,q},0)$ can be written as
$$ L(c_{1,q},h_{i_{1},s_{1}})\otimes L(c_{1,q},h_{i_{2},s_{2}})=\bigoplus_{i\in A_{i_{1},i_{2}}\\ s\in A_{s_{1},s_{2}}} L(c_{1,q},h_{i,s}),$$
where $A_{m,n} =\{m+n-1,m+n-3,\cdots,|m-n|+1\}$ for $m,n>0$.
Many special cases of this result have already been applied in several papers (cf. \cite{am,am1,lin}).

This paper is structured as follows: In Section 2 we give some preliminary results about the representation theory of Virasoro vertex operator algebras.
In Section 3, using the easy part of Frenkel-Zhu's formula,
we get an upper bound for the fusion rules of $L(c_{1,q},0)$. In Section 4 we establish the fusion rules of $L(c_{1,q},0)$ by the limit method.
 In Section 5, we further extend the fusion rules of $L(c_{1,q},0)$ to include some other cases.
Throughout this paper, we assume that the reader is familiar with the axiom theory vertex operator algebras and modules. For more information, see\cite{fhl,ll}.

%%%%%%%%%%%%%%%%%%%%%%%%%%%%%%%%%%%%%%%%%%%%%%%%%%%%%%%%%%%%%%%%%%%%%%%%%%%%%%%%%%%%%%%%%%%%%%%%%%%%%%
\section{vertex operator algebras and modules associated to Virasoro algebra}
In this section, we give a short review of vertex operator algebras and modules associated to Virasoro algebra, details can be found in \cite{ff1,ff2,ik,kr,ll}.
First, recall that Virasoro algebra is the Lie algebra $Vir$ with basis $\{L_{n}\mid n\in\mathbb{Z}\}\cup\{C\}$, satisfying
$$[L_{m},L_{n}]=(m-n)L_{m+n}+\frac{m^{3}-m}{12}\delta_{m+n,0}C$$
and
$$[Vir, C]=\{0\}.$$
Define the following subalgebras of $Vir$:
 $$Vir^{\pm}=\coprod_{\pm n>0}\mathbb{C}L_{n};\  Vir^{0}=L_{0}\oplus C;$$
 $$Vir^{\geq0}=Vir^{+}\oplus Vir^{0};\   Vir^{\geq-1}=Vir^{\geq0}\oplus\mathbb{C}L_{-1}. $$
Let $c$ and $h$ be two complex numbers and let $\mathbb{C}v_{c,h}$ be the one dimensional $Vir^{\geq0}$--module
 with $C$ and $L_{0}$ acting as the scalars $c$ and $h$, and with $Vir^{+}$ acting trivially.
 Set $$M(c,h)=U(Vir)\otimes_{U(Vir^{\geq0})}\mathbb{C}v_{c,h}$$ and call it the Verma module with central charge $c$ and highest weight $h$.
 For any $\mathbb{I}=(1^{r_{1}}2^{r_{2}}\cdots n^{r_{n}})\in\mathcal {P}_{n}$, set
$$e_{\mathbb{I}}=L^{r_{n}}_{-n}\cdots L^{r_{2}}_{-2}L^{r_{1}}_{-1}\in U(Vir^{-})_{-n}.$$
Then, $\{e_{\mathbb{I}}\cdot v_{c,h}|\mathbb{I}\in\mathcal {P}_{n}\}$ forms a basis of the weight subspace $M(c,h)_{h+n}$.
Let $M'(c,h)$ be the largest proper submodule of  $M(c,h)$. Then $L(c,h)=M(c,h)/M'(c,h)$ is an irreducible $Vir$--module.

 We recall the following proposition.
\begin{proposition}\label{rw}
(cf.\cite{rw})Set $n=\alpha\beta$, $c=c(t)=13-6t-6t^{-1}$ and
$h=h_{\alpha,\beta}(t)=\frac{1}{4}(\alpha^{2}-1)t-\frac{1}{2}(\alpha\beta-1)+\frac{1}{4}(\beta^{2}-1)t^{-1}$
for $\alpha,\beta\in\mathbb{Z}_{>0}$. Then there exists
$$S_{n}=\sum_{\mathbb{I}\in\mathcal {P}_{n}}f_{\mathbb{I}}(c,h)e_{\mathbb{I}}\in U(Vir^{-})_{-n}$$
 such that $S_{n}v_{c,h}\in M(c,h)_{h+n}$ is a singular vector, where $f_{\mathbb{I}}(x,y)\in \mathbb{C}[x,y]$
  and $f_{\mathbb{I}_{0}}(x,y)=1$ for $\mathbb{I}_{0}=(1^{n})$.
\end{proposition}

Let $p$, $q$, $r$ and $s$ be positive integers, satisfying $(p,q)=1$, $r<p$ and $s<q$.
Let $c=c_{p,q}=13-6(\frac{q}{p}+\frac{p}{q})$ and $h=h_{p,q;r,s}=\frac{(sp-rq)^{2}-(p-q)^{2}}{4pq}$.
Then by Proposition \ref{rw}, $M(c,h)$ has two singular vectors $u_{p,q}^{r,s}$ and $v_{p,q}^{r,s}$, of weights $h+rs$
and $h+(p-r)(q-s)$ respectively. Moreover, the maximal submodule of $M(c,h)$ is
generated by $u_{p,q}^{r,s}$ and $v_{p,q}^{r,s}$.

Similarly, when $p=1$ and $h= h_{i,s}=\frac{(iq-s)^{2}-(q-1)^{2}}{4q}$ for some $i>0$, $0< s\leq q$,
$M(c_{1,q},h)$ has a singular vector of weight $ h+is$
which generates the maximal proper submodule of $M(c_{1,q},h)$. Moreover, $M(c_{1,q},h)$ is irreducible
when $h\neq h_{i,s}=\frac{(iq-s)^{2}-(q-1)^{2}}{4q}$ for any $i>0$, $0< s\leq q$.

Now consider $\mathbb{C}$ as an $Vir^{\geq-1}$--module with $C$ acting as the scalar $c$,
and with $Vir^{+}\oplus L_{0}\oplus L_{-1}$ acting trivially.
 Set $$V_{c}=U(Vir)\otimes_{U(Vir^{\geq-1})}\mathbb{C}.$$
  Then, it is well known that $V_{c}$ has a canonical structure of vertex operator algebra
   of central charge $c$ and with  $\omega=L(-2)\textbf{1}$ as conformal vector. In this way, $M(c,h)$ and $L(c,h)$
   are modules for $V_{c}$ viewed as a vertex algebra. Furthermore, $L(c,0)$, as a quotient of $V_{c}$, is a simple vertex operator algebra. Note that
$L(c,0)=V_{c}$ when $c\neq c_{p,q}$, where $p, q>1$ and $(p,q)=1$. If $c= c_{p,q}$ for some $p, q$ as above, then $L(c,0)\neq V_{c}$, and $L(c,h)$
 is a $L(c,0)$--module if and only if $h=h_{p,q;r,s}=\frac{(sp-rq)^{2}-(p-q)^{2}}{4pq}$
 for some positive integer $r,s$ satisfying $r<p$ and $s<q$ (cf.\cite{w}).

Now we introduce the definition of intertwining operator and fusion rule for a triple of modules of vertex operator algebra (cf.\cite{fhl}).

\begin{definition}
Let $W_{1}$, $W_{2}$ and $W_{3}$ be three modules over a vertex operator algebra $V$. A linear map $W_{1}\otimes W_{2}\rightarrow W_{3}\{x\}$ or equivalently,
$$W_{1}\rightarrow (Hom(W_{2},W_{3}))\{x\}$$
$$w\mapsto \mathcal {Y}(w,x)=\sum_{n\in\mathbb{Q}}w_{n}x^{-n-1}\ \ \ (where \ w_{n}\in Hom(W_{2},W_{3}))$$
is called an intertwining operator of type
\begin{math}
(\begin{array}{cccc}
\;\; W_{3}  \\
W_{1}\; W_{2}\\
\end{array})
\end{math}
if it satisfies:
\begin{enumerate}
\item[1] (The truncation property)For any $w_{1}\in W_{1}$, $w_{2}\in W_{2}$, $(w_{1})_{n}w_{2}=0$ for $n$ sufficiently large;

\item[2] (The $L_{-1}$-derivative formula)For any $w\in W_{1}$,
$$\mathcal {Y}(L_{-1}w,x)=\frac{d}{dx}\mathcal {Y}(w,x);$$

\item[3] (The Jacobi identity) For any $v\in V$ and $w_{1}\in W_{1}$,
$$x_{0}^{-1}\delta(\frac{x_{1}-x_{2}}{x_{0}})Y(v,x_{1})\mathcal {Y}(w_{1},x_{2}) $$$$
-x_{0}^{-1}\delta(\frac{x_{2}-x_{1}}{-x_{0}})\mathcal {Y}(w_{1},x_{2})Y(v,x_{1}) $$$$
=x_{2}^{-1}\delta(\frac{x_{1}-x_{0}}{x_{2}})\mathcal {Y}(Y(v,x_{0})w_{1},x_{2}).$$
\end{enumerate}
\end{definition}

Set
\begin{math}
I(\begin{array}{cccc}
\;\; W_{3}  \\
W_{1}\; W_{2}\\
\end{array})
\end{math}
 to be the vector space of all intertwining operators of type
\begin{math}
(\begin{array}{cccc}
\;\; W_{3}  \\
W_{1}\; W_{2}\\
\end{array})
\end{math}, its dimension $\mathcal {N}^{W_{k}}_{W_{i},W_{j}}$ is called the fusion rule of type \begin{math}
(\begin{array}{cccc}
\;\; W_{3}  \\
W_{1}\; W_{2}\\
\end{array})
\end{math}.

 The main result of this paper is the following:

 \begin{theorem}\label{main}
Let $i_{n}>0$, $0<s_{n}\leq q$ (n=1,2,3), and $A_{m,n} =\{m+n-1,m+n-3,\cdots,|m-n|+1\}$. Then
\[
\mathcal {N}^{L(c_{1,q},h_{i_{3},s_{3}})}_{L(c_{1,q},h_{i_{1},s_{1}}),L(c_{1,q},h_{i_{2},s_{2}})}\leq1.
\]
The equality hold if and only if $h_{i_{3},s_{3}}=h_{i,s}$ for some $i\in A_{i_{1},i_{2}}$ and\\
$s\in A_{s_{1},s_{2}}$.
\end{theorem}

%%%%%%%%%%%%%%%%%%%%%%%%%%%%%%%%%%%%%%%%%%%%%%%%%%%%%%%%%%%%%%%%%%%%%%%%%%%%%%%%%%%%%%%%%%%%%%%%%%%%%%%

\section{Frenkel-Zhu's formula}
Recall that to a vertex operator algebra $V$, we can associate the Zhu's algebra $A(V)$, and for each
 lowest weight $V$-module $M$, the lowest weight space $M(0)$ has a natural structure of $A(V)$-module.
 More generally, for each $V$-module $M$,
 define $O(M)\subset M$ to be the linear span of elements of type
 $$Res_{z}(Y(a,z)\frac{(1+z)^{deg\ a}}{z^{2}}m)$$
 where $a\in V$ and $m\in M$, and let $A(M)$ be the quotient space $M/O(M)$, then $A(M)$ has a natural structure of an $A(V)$-bimodule.
We recall the following useful result (cf.\cite{fz}).
\begin{proposition}\label{hehe}
For each submodule $M_{1}$ of $M$, $A(M_{1})$ is a submodule of
 the $A(V)$-bimodule $A(M)$, and the quotient $A(M)/A(M_{1})$ is isomorphic to the bimodule $A(M/M_{1})$.
 \end{proposition}
   In the case of Virasoro vertex operator algebras and Verma modules, we have the following results (cf.\cite{fz,ik}).

\begin{proposition} \label{dd}
Let $\mathcal {L}$ to be the subalgebra of $Vir^{-}$ spanned by $$L_{-n-2}+2L_{-n-1}+L_{-n},$$
for $n\geq1$. Then $O(V_{c})=\mathcal {L}.V_{c}$ and $A(V_{c})\cong H_{0}(\mathcal {L}, V_{c})$.
In the case of $M(c,h)$ (resp. the irreducible quotient $L(c,h)$), we also have
$$O(M(c,h))=\mathcal {L}.M(c,h)$$
$$(resp. O(L(c,h))=\mathcal {L}.L(c,h))$$
 and
  $$A(M(c,h))\cong H_{0}(\mathcal {L}, M(c,h)).$$
 $$(resp. A(L(c,h))\cong H_{0}(\mathcal {L}, L(c,h))).$$
\end{proposition}

\begin{proposition} \label{ee}
We have an isomorphism of associative algebra:
$$A(V_{c})\cong\mathbb{C}[x]; \ \ \ \  [\omega]^{n}\mapsto x^{n},\ \ n\in\mathbb{Z}_{\geq0}. \ $$
For Verma module $M(c,h)$, the $A(V_{c})$-bimodule $A(M(c,h))$ is isomorphic to $\mathbb{C}[x,y]$,
where the highest weight vector $v_{c,h}$ represents $\textbf{1}\in \mathbb{C}[x,y]$, and the left and the right actions of $A(V_{c})$ are given by
$$x\cdot f(x,y)=xf(x,y),$$
$$f(x,y)\cdot x=yf(x,y),$$ for any $f(x,y)\in \mathbb{C}[x,y]$.
\end{proposition}

\begin{proposition} \label{ff}
The left and right actions of $A(V_{c})$ on $A(M(c,h))$ are given by
$$[\omega][v]=[(L_{-2}+2L_{-1}+L_{0})v],$$
$$[v][\omega]=[(L_{-2}+L_{-1})v],$$
for any $v\in M(c,h) $, where $\omega=L_{-2}\textbf{1}$.
\end{proposition}

From now on, $W_{i}=\oplus_{n\in\mathbb{N}}W_{i}(n)$ $(i=1,2,3)$ will always be irreducible $V$-modules,
 where $W_{i}(n)$ is the $L_{0}$-eigenspace of $W_{i}$ with eigenvalue $n+h_{i}$.
\begin{proposition}\label{in}
(cf.\cite{fz})Let $\mathcal {Y}(\cdot,x)$ be an intertwining operator of type
\begin{math}
(\begin{array}{cccc}
\;\; W_{3}  \\
W_{1}\; W_{2}\\
\end{array})
\end{math}.
Then $\mathcal {Y}(\cdot,x)$ has the following form:
$$\mathcal {Y}(w,x)=\sum_{n\in\mathbb{Z}}w(n)x^{-n-1}x^{-h_{1}-h_{2}+h_{3}},$$
such that for any $w\in W_{1}(k)$
$$w(n)W_{2}(m)\subset W_{3}(m+k-n-1)$$
\end{proposition}

We need also the symmetry property of fusion rules,
 i.e., $\mathcal {N}^{W_{3}}_{W_{1},W_{2}}=\mathcal {N}^{W_{3}}_{W_{2},W_{1}}$ (cf.\cite{fhl}).

Let $\mathcal {Y}(\cdot,x)$ be an intertwining operator of type
\begin{math}
(\begin{array}{cccc}
\;\; W_{3}  \\
W_{1}\; W_{2}\\
\end{array})
\end{math}.
 By Proposition \ref{in}, we can define a linear map $o_{\mathcal {Y}}$ from $W_{1}\otimes W_{2}(0)$ to $W_{3}(0)$ by sending
  $w_{1}\otimes w_{2}$ ($w_{1}\in W_{1}(n)$, $w_{2}\in W_{2}(0)$) to $w_{1}(n-1)w_{2}$. It can be proved that
   $w_{1}(n-1)w_{2}=0$ for $w_{1}\in O(W_{1})$, and $o_{\mathcal {Y}}$ induces an $A(V)$-homomorphism $$\pi(\mathcal {Y}):A(W_{1})\otimes_{A(V)}W_{2}(0)\rightarrow W_{3}(0).$$
 Thus we get a linear map:
 $$\pi:
I(\begin{array}{cccc}
\;\; W_{3}  \\
W_{1}\; W_{2}\\
\end{array})
\rightarrow Hom_{A(V)}(A(W_{1})\otimes_{A(V)}W_{2}(0), W_{3}(0))$$
The Frenkel-Zhu's formula  (cf.\cite{fz}) states that $\pi$ is an isomorphism if $W_{i}$ $(i=1,2,3)$ are irreducible modules.
It was pointed out in \cite{li} that this formula only holds for rational vertex operator algebras, and for more general vertex
 operator algebras, we have the following proposition (cf.\cite{li}).
\begin{proposition}\label{ffff}
If $W_{3}$ is irreducible, then $$\pi:
I(\begin{array}{cccc}
\;\; W_{3}  \\
W_{1}\; W_{2}\\
\end{array})
\rightarrow Hom_{A(V)}(A(W_{1})\otimes_{A(V)}W_{2}(0), W_{3}(0))$$ is injective.
\end{proposition}

Now we follow the treatment in \S 9.3 of \cite{ik}. First, we consider the three $L(c_{1,q},0)$-modules $L(c_{1,q},h_{i_{n},s_{n}})$,
where $i_{n}>0$, $0<s_{n}\leq q$ (n=1,2,3). We want to compute the dimension of
$$Hom_{A(L(c_{1,q},0))}(A(L(c_{1,q},h_{i_{1},s_{1}}))\otimes_{A(L(c_{1,q},0))}L(c_{1,q},h_{i_{2},s_{2}})(0), L(c_{1,q},h_{i_{3},s_{3}})(0)).$$
By Proposition \ref{dd}, its dual space is isomorphic to the simultaneous eigenspace of the left and right actions
of $[\omega]$ on $H_{0}(\mathcal {L}, L(c_{1,q},h_{i_{1},s_{1}}))^{\ast}=H^{0}(\mathcal {L}, L(c_{1,q},h_{i_{1},s_{1}})^{\ast})$
with the eigenvalues $-h_{i_{3},s_{3}}$ and $-h_{i_{2},s_{2}}$ respectively;
denote this eigenspace by $H^{0}(\mathcal {L}, L(c_{1,q},h_{i_{1},s_{1}})^{\ast})^{(-h_{i_{3},s_{3}}, -h_{i_{2},s_{2}})}$. Then the surjection
$$M(c_{1,q},h_{i_{1},s_{1}})\twoheadrightarrow L(c_{1,q},h_{i_{1},s_{1}})$$
induces an injection
$$i:H^{0}(\mathcal {L}, L(c_{1,q},h_{i_{1},s_{1}})^{\ast})^{(-h_{i_{3},s_{3}},-h_{i_{2},s_{2}})}\hookrightarrow H^{0}(\mathcal {L}, M(c_{1,q},h_{i_{1},s_{1}})^{\ast})^{(-h_{i_{3},s_{3}},-h_{i_{2},s_{2}})}.$$
The argument in \S 9.3 of \cite{ik} shows that
$H^{0}(\mathcal {L}, M(c_{1,q},h_{i_{1},s_{1}})^{\ast})^{(-h_{i_{3},s_{3}},-h_{i_{2},s_{2}})}$ is one-dimensional, and $i$ is an isomorphism
 if and only if
\begin{equation}\label{eq}
P_{i_{1},s_{1}}(-h_{i_{2},s_{2}}, -h_{i_{3},s_{3}}+h_{i_{1},s_{1}},q)=0,
\end{equation}
 where $P_{\alpha,\beta}(a,b;\xi)\in\mathbb{C}[a,b,\xi,\xi^{-1}]$ satisfies

  $$P_{\alpha,\beta}(a,b;\xi)^{2}=\prod_{k=0}^{\alpha-1}\prod_{l=0}^{\beta-1}Q_{k,l}^{\alpha,\beta}(a,b;\xi),$$
\begin{eqnarray*}
&&Q_{k,l}^{\alpha,\beta}(a,b;\xi)\\
&&=[(b-a)-(k\xi^{\frac{1}{2}}-l\xi^{-\frac{1}{2}})\{  (\alpha-k)\xi^{\frac{1}{2}}-(\beta-l)\xi^{-\frac{1}{2}} \}] \\
&&\times[(b-a)-\{(k+1)\xi^{\frac{1}{2}}-(l+1)\xi^{-\frac{1}{2}}\}\{  (\alpha-k-1)\xi^{\frac{1}{2}}-(\beta-l-1)\xi^{-\frac{1}{2}} \}] \\
&&+\{ (\alpha-2k-1)\xi^{\frac{1}{2}}-(\beta-l-1)\xi^{-\frac{1}{2}} \}^{2}a.
\end{eqnarray*}

Direct computation shows that Equation (\ref{eq}) is equivalent to the equation
$$\prod_{k=0}^{i_{1}-1}\prod_{l=0}^{s_{1}-1}(h_{i_{3},s_{3}}-h_{i_{1}+i_{2}-2k-1,s_{1}+s_{2}-2l-1})=0$$
Now combining the symmetry property of fusion rules and Proposition \ref{ffff} yields
 $\mathcal {N}^{L(c_{1,q},h_{i_{3},s_{3}})}_{L(c_{1,q},h_{i_{1},s_{1}}),L(c_{1,q},h_{i_{2},s_{2}})}\leq1$,
 and $\mathcal {N}^{L(c_{1,q},h_{i_{3},s_{3}})}_{L(c_{1,q},h_{i_{1},s_{1}}),L(c_{1,q},h_{i_{2},s_{2}})}=1$
 only if the following two equations hold:
$$\prod_{k=0}^{i_{1}-1}\prod_{l=0}^{s_{1}-1}
(h_{i_{3},s_{3}}-h_{i_{1}+i_{2}-2k-1,s_{1}+s_{2}-2l-1})=0;$$
$$\prod_{k=0}^{i_{2}-1}\prod_{l=0}^{s_{2}-1}
(h_{i_{3},s_{3}}-h_{i_{1}+i_{2}-2k-1,s_{1}+s_{2}-2l-1})=0.$$
For each $m,n>0$, set $A_{m,n} =\{m+n-1,m+n-3,\cdots,|m-n|+1\}$.
Then one checks that these two equations are equivalent to the existences of
$i\in A_{i_{1},i_{2}}$, $s\in A_{s_{1},s_{2}}$ such that $h_{i_{3},s_{3}}=h_{i,s}$.

To sum up, we have proved that $\mathcal {N}^{L(c_{1,q},h_{i_{3},s_{3}})}_{L(c_{1,q},h_{i_{1},s_{1}}),L(c_{1,q},h_{i_{2},s_{2}})}\leq1$
 if there exists $i\in A_{i_{1},i_{2}}$, $s\in A_{s_{1},s_{2}}$ such that $h_{i_{3},s_{3}}=h_{i,s}$, and $\mathcal {N}^{L(c_{1,q},h_{i_{3},s_{3}})}_{L(c_{1,q},h_{i_{1},s_{1}}),L(c_{1,q},h_{i_{2},s_{2}})}=0$ otherwise.

\section{Construction of intertwining operators}

In this section we always adopt the following convention:
\begin{convention}\label{con} We always fix a nonzero highest weight vector $v_{c,h}$ in the Verma module $M(c,h)$ for each $c$ and $h$,
and identify $U(Vir^{-})$ with $M(c,h)$ by sending $e_{\mathbb{I}}$ to $e_{\mathbb{I}} v_{c,h}$.
\end{convention}

 We fix three $L(c_{1,q},0)$-modules $L(c_{1,q},h_{i_{n},s_{n}})$ ($n=1,2,3$), where $i_{n}>0$, $0< s_{n}\leq q$,
and assume that there exists $i\in A_{i_{1},i_{2}}$, $s\in A_{s_{1},s_{2}}$, satisfying $h_{i_{3},s_{3}}=h_{i,s}$.
The purpose of this section is to construct a nonzero intertwining operator of type
(\begin{math}
\begin{array}{cccc}
\;\; L(c_{1,q},h_{i_{3},s_{3}})  \\
L(c_{1,q},h_{i_{1},s_{1}})\; L(c_{1,q},h_{i_{2},s_{2}})\\
\end{array}
\end{math}).
Set $c_{k}=c_{k,kq-1},$ $h_{n}^{k}=\frac{(i_{n}(kq-1)-s_{n}k)^{2}-(kq-1-k)^{2}}{4k(kq-1)}$.
By the fusion rules of $L(c_{k},0)$ when $k$ is large enough, 
there exists a nontrivial intertwining operator $\mathcal {Y}_{k}(\cdot, x)$ of type
\begin{math}
(\begin{array}{cccc}
\;\; L(c_{k},h_{3}^{k})  \\
L(c_{k},h_{1}^{k})\; L(c_{k},h_{2}^{k})\\
\end{array})
\end{math}.
Our method is to get the desired intertwining operator from the limit of $\mathcal {Y}_{k}(\cdot, x)$
 as $k$ approaches infinity. Hence from now on, we always assume that $k$ is large enough when needed.
 We say that a sequence of monomials $\{a_{k}x^{n_{k}}\}$ converges to the limit $ax^{n}$
 $$\lim_{k\rightarrow\infty}a_{k}x^{n_{k}}=ax^{n}$$
  if $\{a_{k}\}$
 and $\{n_{k}\}$ converge to the limits $a$ and $n$ respectively.
The following proposition is crucial for our construction.
 \begin{proposition}\label{finite}
As a left $A(L(c_{k},0))$-module, $A(L(c_{k},h_{n}^{k}))$ is generated by
$[v_{c_{k},h_{n}^{k}}],[L_{-1}v_{c_{k},h_{n}^{k}}],\cdots,[L_{-1}^{i_{n}s_{n}-1}v_{c_{k},h_{n}^{k}}]$.
\end{proposition}
\begin{proof}
Combining Proposition \ref{dd}, \ref{ee} and \ref{ff} implies the formula
$[L_{-n}v]=(-1)^{n}(ny-x+wt(v))[v]$ in $A(M(c_{k},h_{n}^{k}))$ for each homogenous $v\in L(c_{k},h_{n}^{k})$
(recall the identification $A(M(c_{k},h_{n}^{k}))=\mathbb{C}[x,y]$ in Proposition \ref{ee}).
From this formula and Proposition \ref{rw} we have
$$[L_{-1}^{m}v_{c_{k},h_{n}^{k}}]=(x-y)^{m}+ lower\ terms$$ and
$$[S_{i_{n}s_{n}}v_{c_{k},h_{n}^{k}}]=[v_{i_{n},s_{n}}]=(x-y)^{i_{n}s_{n}}+ lower\ terms$$
in $A(M(c_{k},h_{n}^{k}))\cong\mathbb{C}[x,y]$.
Now by Proposition \ref{hehe}, $[S_{i_{n}s_{n}}v_{c_{k},h_{n}^{k}}]$ lies in the kernel of the surjective morphism
 $A(M(c_{k},h_{n}^{k}))\twoheadrightarrow A(L(c_{k},h_{n}^{k}))$, hence $A(L(c_{k},h_{n}^{k}))$ can be generated,
 as a left $A(L(c_{k},0))$-module by
$$[v_{c_{k},h_{n}^{k}}],[L_{-1}v_{c_{k},h_{n}^{k}}],\cdots,[L_{-1}^{i_{n}s_{n}-1}v_{c_{k},h_{n}^{k}}]$$.
\end{proof}

Now we are well prepared for the construction.
It suffices to construct a
 bilinear pair $\lceil\cdot,\cdot\rfloor$ (with value in $\mathbb{C}\{x\}$) between
 $(L(c_{1,q},h_{i_{3},s_{3}}))^{\ast}$ and $L(c_{1,q},h_{i_{1},s_{1}})\otimes L(c_{1,q},h_{i_{2},s_{2}})$
  that satisfies the corresponding properties. The construction is divided into several steps.
  In the following using Convention \ref{con}, we always identify Verma modules $M(c,h)$ for different pairs of $\{c,h\}$.
  For a $Vir$--module $M$ of lowest weight $h$, we always use $M(n)$ to denote the weight subspace of weight $h+n$. The same notations are applied
  to submodules, quotient modules, and dual modules of $M$.

 \textbf{ Step 1}. Let $v'\in L(c_{1,q},h_{i_{3},s_{3}})^{\ast}(0)=M(c_{1,q},h_{i_{3},s_{3}})^{\ast}(0)=M(c_{k},h_{3}^{k})^{\ast}(0)$
 be defined by $v'(v_{c_{1,q},h_{i_{3},s_{3}}})=1$.
 For each homogenous $v_{1}\in M(c_{1,q},h_{i_{1},s_{1}})$
   and $v_{2}=v_{c_{1,q},h_{i_{2},s_{2}}}=v_{c_{k},h_{2}^{k}}$, $\lceil v',v_{1}\otimes v_{2}\rfloor$ is defined as follows:

  Set $a_{1}= v_{c_{1,q},h_{i_{1},s_{1}}}, a_{2}=L_{-1}v_{c_{1,q},h_{i_{1},s_{1}}},\cdots,
  a_{i_{1}s_{1}}=L_{-1}^{i_{1}s_{1}-1}v_{c_{1,q},h_{i_{1},s_{1}}}$
   in $M(c_{1,q},h_{i_{1},s_{1}})=M_{c_{k},h_{1}^{k}}$.
Then for each $k$, there is some $i$ such that $\langle v',\mathcal {Y}_{k}(a_{i}^{k}, x) v_{2}\rangle\neq0$,
otherwise, by Proposition \ref{ffff} and \ref{finite}, $\mathcal {Y}_{k}$ will be zero.
As $\langle v',\mathcal {Y}_{k}(a_{i+1}, x) v_{2}\rangle$
is the derivation of $\langle v',\mathcal {Y}_{k}(a_{i}, x) v_{2}\rangle$, hence
 $\langle v',\mathcal {Y}_{k}(a_{1}, x) v_{2}\rangle\neq0$ for each $k$.
By replacing each $\mathcal {Y}_{k}$ by a nonzero multiple,
there exists a subsequence of intertwining operators $\{\mathcal {Y}_{n_{k}}\}$  such that
$$\langle v',\mathcal {Y}_{n_{1}}(a_{1}, x)v_{2}\rangle, \langle v',\mathcal {Y}_{n_{2}}(a_{1}, x)v_{2}\rangle,
\cdots, \langle v',\mathcal {Y}_{n_{k}}(a_{1}, x) v_{2}\rangle, \cdots$$
converge to a nonzero monomial.
Now assume that the sequence
$$\langle v',\mathcal {Y}_{n_{1}}(v, x)v_{2}\rangle, \langle v',\mathcal {Y}_{n_{2}}(v, x)v_{2}\rangle,
\cdots, \langle v',\mathcal {Y}_{n_{k}}(v, x) v_{2}\rangle, \cdots$$
converges for a homogeneous $v\in M(c_{1,q},h_{i_{1},s_{1}})$.
%, where we also consider $v$ as element of $M(c_{k},h_{n}^{k})$ by Convention \ref{con}
Then the same result holds for $L_{0}v$ and $L_{-1}v$.
As the left action of $A(V_{c})$ on $A(M(c,h))$ is given by
$$[\omega][v]=[(L_{-2}+2L_{-1}+L_{0})v],$$
by the construction of the linear map $\pi$ in \S 3, we see that
$$\langle v',\mathcal {Y}_{n_{1}}(L_{-2}v, x)v_{2}\rangle, \langle v',\mathcal {Y}_{n_{2}}(L_{-2}v, x)v_{2}\rangle,
\cdots, \langle v',\mathcal {Y}_{n_{k}}(L_{-2}v, x) v_{2}\rangle, \cdots$$
also converges.
By induction and the equality $[L_{-n},L_{-1}]=(1-n)L_{-n-1}$,
the same is true for $L_{-n}v$ $(n>2)$.
As $M(c_{1,q},h_{i_{1},s_{1}})$ is generated by $a_{1}= v_{c_{1,q},h_{i_{1}}}$, we conclude by induction that for each homogeneous $v_{1}\in M(c_{1,q},h_{i_{1},s_{1}})$, the sequence
$$\langle v',\mathcal {Y}_{n_{1}}(v_{1}, x)v_{2}\rangle, \langle v',\mathcal {Y}_{n_{2}}(v_{1}, x)v_{2}\rangle,
\cdots, \langle v',\mathcal {Y}_{n_{k}}(v_{1}, x) v_{2}\rangle, \cdots$$
converges. Let $\lceil v',v_{1}\otimes v_{2}\rfloor$ be the limit
and Step 1 is complete.

\textbf{Step 2}. For any $v_{1}\in M(c_{1,q},h_{i_{1},s_{1}})$
   and $v_{2}\in M(c_{1,q},h_{i_{2},s_{2}})$ we want to define $\lceil v',v_{1}\otimes v_{2}\rfloor$ as the limit of the sequence
\begin{equation} \label{seq}
\langle v',\mathcal {Y}_{n_{1}}(v_{1}, x)v_{2}\rangle, \langle v',\mathcal {Y}_{n_{2}}(v_{1}, x)v_{2}\rangle,
\cdots, \langle v',\mathcal {Y}_{n_{k}}(v_{1}, x) v_{2}\rangle, \cdots
\end{equation}
 Thus we need to show that the limit of the sequence (\ref{seq}) exists for each $v_{1}\in M(c_{1,q},h_{i_{1},s_{1}})$
   and $v_{2}\in M(c_{1,q},h_{i_{2},s_{2}})$. Step 1 shows that when $v_{2}=v_{c_{1,q},h_{i_{2},s_{2}}}$, the limit exists.
   Thus by induction, it suffices to prove that if the limit of the sequence (\ref{seq}) exists for any $v_{1}\in M(c_{1,q},h_{i_{1},s_{1}})$
   and a fixed $v_{2}\in M(c_{1,q},h_{i_{2},s_{2}})$, then the same is true with $v_{2}$ replaced by $L_{n}v_{2}$ $(n>0)$.
  But this follows directly from the following identity:
  \begin{eqnarray*}
  \langle v',\mathcal {Y}_{k}(v_{1}, x) L_{n}v_{2}\rangle=\langle v',L_{n}\mathcal {Y}_{k}(v_{1}, x) v_{2}\rangle\\
 -\sum_{i=0}^{\infty}{n+1 \choose i}x^{n+1-i}\langle v',\mathcal {Y}_{k}(L_{i-1}v_{1}, x) v_{2}\rangle\\
 = -\sum_{i=0}^{\infty}{n+1 \choose i}x^{n+1-i}\langle v',\mathcal {Y}_{k}(L_{i-1}v_{1}, x) v_{2}\rangle.
  \end{eqnarray*}
Hence we conclude that the limit of the sequence (\ref{seq}) exists for any homogeneous $v_{2}\in M(c_{1,q},h_{i_{1},s_{1}})$.
 Set $$\lceil v',v_{1}\otimes v_{2}\rfloor=\lim_{k\rightarrow\infty}\langle v',\mathcal {Y}_{n_{k}}(v_{1}, x) v_{2}\rangle,$$
 and Step 2 is complete.

 \textbf{Step 3}. Now we want to define $\lceil v'_{3},v_{1}\otimes v_{2}\rfloor$ for any homogeneous $v_{1}\in M(c_{1,q},h_{i_{1},s_{1}})$,
   $v_{2}\in M(c_{1,q},h_{i_{2},s_{2}})$ and $v'_{3}\in L(c_{1,q},h_{i_{3},s_{3}})^{*}\subset M(c_{1,q},h_{i_{3},s_{3}})^{*}$.

\begin{lemma}\label{lili}
 For any $\mathbb{I}\in\mathcal {P}_{n}$,
 $\langle e_{\mathbb{I}}v',\mathcal{Y}_{n_{k}}(v_{1}, x)v_{2}\rangle$ converges to a finite limit as  $k$ approaches infinity.
\end{lemma}

\begin{proof}
By induction on the length of $\mathbb{I}$, this lemma follows directly from the formula
$$\langle L_{n}w',w\rangle=\langle w',L_{-n}w\rangle$$
for any $w'\in L(c_{k},h_{3}^{k})^{\ast}$ and $w\in L(c_{k},h_{3}^{k})$.
\end{proof}

   By Convention \ref{con}, we can identify $M(c_{1,q},h_{i_{3},s_{3}})^{*}$ with $M(c_{k},h_{3}^{k})^{*}$.
   Under this identification, $L(c_{k},h_{3}^{k})^{*}(n)$ converges to $L(c_{1,q},h_{i_{3},s_{3}})^{*}(n)$ as $k$ approaches infinity.
    As $L(c_{1,q},h_{i_{3},s_{3}})$ is irreducible, $L(c_{1,q},h_{i_{3},s_{3}})^{*}$
    is generated by $L(c_{1,q},h_{i_{3},s_{3}})^{*}(0)=\mathbb{C}v'$ as a module over $Vir^{-}$,
    thus we can choose a subset $\{\mathbb{I}_{1},\cdots,\mathbb{I}_{s}\} $  of $\mathcal {P}_{n}$ such that
     $e_{\mathbb{I}_{1}}v',\cdots,e_{\mathbb{I}_{s}}v'$ form a basis of $L(c_{1,q},h_{i_{3},s_{3}})^{*}(n)$.
     By our convention, both $L(c_{1,q},h_{i_{3},s_{3}})^{*}(n)$ and $L(c_{t_{k}},h_{3}^{k})^{*}(n)$
      are subspaces of $M(c_{1,q},h_{i_{3},s_{3}})^{*}(n)=M(c_{k},h_{3}^{k})^{*}(n)$.
      Moreover, by Proposition \ref{rw}, $$\lim_{k\rightarrow\infty}L(c_{t_{k}},h_{3}^{k})^{*}(n)=L(c_{1,q},h_{i_{3},s_{3}})^{*}(n).$$
      Hence it is easy to see that $ e_{\mathbb{I}_{1}}v',\cdots,e_{\mathbb{I}_{s}}v' $, as elements of $M(c_{k},h_{3}^{k})^{*}$,
      form a basis of $L(c_{t_{k}},h_{3}^{k})^{*}(n)$,
       and converge to $e_{\mathbb{I}_{1}}v',\cdots,e_{\mathbb{I}_{s}}v'$ respectively in $M(c_{1,q},h_{i_{3},s_{3}})^{*}$
       as $k$ approaches infinity.

    Now for any homogeneous $v'_{3}\in L(c_{1,q},h_{i_{3},s_{3}})^{*}(n)$
    we can choose a $v'_{3,k}\in L(c_{k},h_{3}^{k})^{\ast}(n)$, such that the sequence
    $$\cdots v'_{3,k}, v'_{3,k+1}, v'_{3,k+2}\cdots $$
    converges to $v'_{3}$. If we write $$v'_{3}= a_{1}e_{\mathbb{I}_{1}}v'+\cdots+a_{s}e_{\mathbb{I}_{s}}v',$$ and
    $$v'_{3,k}= a_{1,k}e_{\mathbb{I}_{1}}v'+\cdots+a_{s,k}e_{\mathbb{I}_{s}}v',$$ then, for each $i$ the sequence
    $$\cdots a_{i,k}, a_{i,k+1}, a_{i,k+2}\cdots$$ converges to $a_{i}$. By Lemma \ref{lili} we can set
    $$\lceil v'_{3},v_{1}\otimes v_{2}\rfloor=\lim_{k\rightarrow\infty}\langle v'_{3,n_{k}},\mathcal {Y}_{n_{k}}(v_{1}, x) v_{2}\rangle,$$
    and it is easy to see that this setting is independent of the choice of the sequence
    $$\cdots v'_{3,k}, v'_{3,k+1}, v'_{3,k+2}\cdots $$

 \textbf{Step 4}. Now we check that the pairing $\lceil\cdot,\cdot\rfloor$ induces an intertwining operator of type
 \begin{math}
(\begin{array}{cccc}
\;\; L(c_{1,q},h_{i_{3},s_{3}})  \\
M(c_{1,q},h_{i_{1},s_{1}})\; M(c_{1,q},h_{i_{2},s_{2}})\\
\end{array})
\end{math}.

It suffices to verify the $L_{-1}$-derivative formula
   $$\frac{d}{dx}\lceil v'_{3},v_{1}\otimes v_{2}\rfloor=\lceil v'_{3},L_{-1}v_{1}\otimes v_{2}\rfloor$$

and the Jacobi identity
 \begin{eqnarray*} \label{jac}
 x_{0}^{-1}\delta(\frac{x_{1}-x}{x_{0}}) \lceil  Y(e^{x_{1}L_{1}}(-x_{1}^{-2})^{L_{0}}v,x_{1}^{-1})v'_{3},v_{1}\otimes v_{2}\rfloor\\
-x_{0}^{-1}\delta(\frac{x-x_{1}}{-x_{0}})\lceil v'_{3}, v_{1}\otimes Y(v,x_{1})v_{2}\rfloor\\
=x^{-1}\delta(\frac{x_{1}-x_{0}}{x})\lceil v'_{3}, Y(v,x_{1})v_{1}\otimes v_{2}\rfloor,
\end{eqnarray*}
where $v\in V_{c_{1,q}}$
(this form of Jacobi identity follows from the graded dual module structure of $L(c_{1,q},h_{i_{3},s_{3}})^{*}$ defined in \S 5.2 of \cite{fhl}),
 and the truncation property follows as a consequence.

The $L_{-1}$-derivative formula follows directly from our definition of $\lceil\cdot,\cdot\rfloor$
and the fact that the derivation $\frac{d}{dx}$ commutes with the limiting operation.
In order to prove the Jacobi identity, we identify $V_{c_{1,q}}$ with $V_{c_{k}}$ by linear isomorphism
$V_{c_{1,q}}\rightarrow V_{c_{k}}$ which sending $\textbf{1}$ to $\textbf{1}$ and commutates with the action of $Vir^{-}$.
Then, it is easy to see that the coefficients of
$Y(v,x_{0})v_{1}$ (resp. $Y(v,x_{1})v_{2}$), as elements of $M(c_{k},h_{1}^{k})$ (resp. $M(c_{k},h_{2}^{k})$), converge to the corresponding coefficients of
$Y(v,x_{0})v_{1}$ (resp. $Y(v,x_{1})v_{2}$), as elements of $M(c_{1,q},h_{i_{1},s_{1}})$ (resp. $M(c_{1,q},h_{i_{2},s_{2}})$). If we choose,
as in Step 3,
a $v'_{3,k}\in L(c_{k},h_{3}^{k})^{\ast}$ for each $k$, such that the sequence
    $$\cdots v'_{3,k}, v'_{3,k+1}, v'_{3,k+2}\cdots $$
    converges to $v'_{3}$, then the coefficients of $Y(e^{x_{1}L_{1}}(-x_{1}^{-2})^{L_{0}}v,x_{1}^{-1})v'_{3,k}$ converge to the corresponding
    coefficients of $Y(e^{x_{1}L_{1}}(-x_{1}^{-2})^{L_{0}}v,x_{1}^{-1})v'_{3}$. Now the Jacobi identity of $\lceil\cdot,\cdot\rfloor$ follows from the
    Jacobi identities of $\mathcal {Y}_{k}$ and Step 4 is complete.

\textbf{Step 5}.  Show that $\lceil v'_{3},v_{1}\otimes v_{2}\rfloor=0$ if $v_{1}$ lies in the maximal proper submodule $M'(c_{1,q},h_{i_{1},s_{1}})$ of $M(c_{1,q},h_{i_{1},s_{1}})$.
Let $M'(c_{k},h_{1}^{k})$ be the maximal submodule of $M(c_{k},h_{1}^{k})$,
 then it is easy to see that $M'(c_{k},h_{1}^{k})(n)$ converges to $M'(c_{1,q},h_{i_{1},s_{1}})(n)$  for each $n$ as $k$ approaches infinity.
 Thus there exists
 a $v_{1,k}\in M'(c_{k},h_{1}^{k})$ for each $k$, such that the sequence
    $$\cdots v_{1,k}, v_{1,k+1}, v_{1,k+2}\cdots $$ converges to $v_{1}$. Using the argument in Step 3, we see that
     $$\lceil v'_{3},v_{1}\otimes v_{2}\rfloor=\lim_{k\rightarrow\infty}\langle v'_{3,n_{k}},\mathcal {Y}_{n_{k}}(v_{1,n_{k}}, x) v_{2}\rangle$$
     where  $$\cdots v'_{3,k}, v'_{3,k+1}, v'_{3,k+2}\cdots $$ is the sequence converging to $v'_{3}$ in Step 3.
     Since $\mathcal {Y}_{k}(\cdot, x)$ is an intertwining operator of type
\begin{math}
(\begin{array}{cccc}
\;\; L(c_{k},h_{3}^{k})  \\
L(c_{k},h_{1}^{k})\; L(c_{k},h_{2}^{k})\\
\end{array})
\end{math},  this forces $\lceil v'_{3},v_{1}\otimes v_{2}\rfloor=0$.

   \textbf{Step 6}. Show that $\lceil v'_{3},v_{1}\otimes v_{2}\rfloor=0$
    if $v_{2}$ lies in the maximal proper submodule $M'(c_{1,q},h_{i_{2},s_{2}})$ of $M(c_{1,q},h_{i_{2},s_{2}})$.
   It suffices to repeat the argument in Step 5 and we omit the details.

From the above construction we see that the pairing $\lceil\cdot,\cdot\rfloor$ induces a nonzero intertwining operator $\mathcal {Y}(\cdot, x)$ of type
 \begin{math}
(\begin{array}{cccc}
\;\; L(c_{1,q},h_{i_{3},s_{3}})  \\
L(c_{1,q},h_{i_{1},s_{1}})\; L(c_{1,q},h_{i_{2},s_{2}})\\
\end{array})
\end{math} such that $\lceil v'_{3},v_{1}\otimes v_{2}\rfloor=\langle v'_{3},\mathcal {Y}(v_{1}, x)v_{2}\rangle$.
This finishes our construction. Hence the proof of Theorem \ref{main} is complete.

\begin{remark}
The limit method is quite necessary, for we can not construct these intertwining operators by lattice vertex operator algebra as in \cite{dj1,mi}.
We hope to formalize this method and find more applications in future work.
\end{remark}

\section{Further extension}
From Section 2, we see that the Verma module $M(c_{1,q},h)$ is irreducible
if and only if $h\neq h_{i,s}=\frac{(iq-s)^{2}-(q-1)^{2}}{4q}$ for any $i>0$, $0< s\leq q$.
In this section, we consider the fusion rules of the type
 \begin{math}
 (\begin{array}{cccc}
\;\; M(c_{1,q},h')  \\
L(c_{1,q},h_{i_{1},s_{1}})\; M(c_{1,q},h)\\
\end{array})
\end{math}
and
 \begin{math}
 (\begin{array}{cccc}
\;\; M(c_{1,q},h)  \\
L(c_{1,q},h_{i_{1},s_{1}})\; L(c_{1,q},h_{i_{2},s_{2}})\\
\end{array})
\end{math}
where $M(c_{1,q},h)$ and $M(c_{1,q},h')$ are irreducible Verma modules and $i_{1},s_{1}, i_{2},s_{2}$ are as before.
 By Theorem 2.11 in \cite{li}, Frenkel-Zhu's formula holds in the first case. Hence the argument of Section 3 directly implies

 \begin{theorem}\label{main1}
\[
\mathcal {N}^{M(c_{1,q},h')}_{L(c_{1,q},h_{i,s}),M(c_{1,q},h)}\leq1
\]
where $i>0$, $0<s\leq q$,  and $M(c_{1,q},h)$ and $M(c_{1,q},h')$ are irreducible Verma modules.
 Set $h=\frac{s'^{2}-(q-1)^{2}}{4q}$ for some complex number $s'$. Then the equality holds if and only if
$h'=\frac{(jq-s'-t)^{2}-(q-1)^{2}}{4q}$ for some $j\in\{-i+1,-i+3,\cdots,i-1\} $ and $t\in\{-s+1,-i+3,\cdots,s-1\}$
\end{theorem}

Similarly, we have

 \begin{theorem}\label{main12}
\[
\mathcal {N}^{M(c_{1,q},h)}_{L(c_{1,q},h_{i_{1},s_{1}}),L(c_{1,q},h_{i_{2},s_{2}})}=0
\]
where $i_{1},i_{2}>0$, $0<s_{1}, s_{2}\leq q$ and $M(c_{1,q},h)$ is an irreducible Verma module.
\end{theorem}

\noindent {\bf Acknowledgments}

The author would like to express his deep gratitude to A. Milas for
valuable comments and suggestions about this paper.

\end{document}